\documentclass[oneside,english]{amsart}

\usepackage{amssymb}
\usepackage{amscd}
\usepackage{amsfonts, amsmath, amssymb, graphicx, epstopdf,amstext}
\usepackage{amsthm}
\usepackage{chngcntr}
\usepackage{etoolbox}
\usepackage{lipsum}
\numberwithin{equation}{section} 
\numberwithin{figure}{section} 
\theoremstyle{plain}
\newtheorem*{thm*}{Theorem}
\theoremstyle{plain}
\newtheorem{thm}{Theorem}[section]
\theoremstyle{definition}
\newtheorem{defn}[thm]{Definition}
\theoremstyle{plain}
\newtheorem{lem}[thm]{Lemma}
\theoremstyle{plain}

\theoremstyle{plain}

\theoremstyle{remark}

\theoremstyle{remark}

\theoremstyle{remark}

\newtheorem*{acknowledgement*}{Acknowledgement}

\begin{document}

\title[A general version of Beltrami's theorem in Finslerian setting]{A general version of
Beltrami's theorem in Finslerian setting}

\author[Bucataru]{Ioan Bucataru}
\address{Faculty of  Mathematics \\ Alexandru Ioan Cuza University \\ Ia\c si, 
  Romania}
\email{bucataru@uaic.ro}
\urladdr{http://www.math.uaic.ro/\textasciitilde{}bucataru/}
\author[Cre\c tu]{Georgeta Cre\c tu}
\address{Faculty of  Mathematics \\ Alexandru Ioan Cuza University \\ Ia\c si, 
  Romania}
\email{getutza\_cretzu@yahoo.com}

\date{\today}

\begin{abstract}
We present a new proof of a Finslerian version of Beltrami's theorem
(1865) which works also in dimension 2. 
\end{abstract}



\maketitle

\section{Introduction}
Flag curvature in Finsler geometry is a natural extension of sectional
curvature. While in Riemannian geometry,
metrics of constant sectional curvature are well understood and classified, in
Finsler geometry the problem is far from being solved, \cite{Li14, ZSh}. In
Finsler geometry, there are many characterisations for metrics of
constant flag curvature, \cite{AZ06, B29, BC18, BF19}. In this paper, in
Theorem \ref{CONSTANT:FC}, we
characterise Finsler metrics of constant
curvature in terms of three conditions, which we call
CC-conditions: \eqref{isotropic1}, \eqref{dJ:xi},
\eqref{dh:xi}. 
There are two motivations for this formulation of the
three CC-conditions. First, we can view Theorem \ref{CONSTANT:FC} as
a Finslerian extension of Schur's lemma that includes dimension $2$ as
well. Secondly, it is easy to check whether or not each of the three
CC-conditions are projectively invariant.

The first and third CC-condition are invariant under projective deformations and
they have corresponding quantities in the Riemannian context. The
second CC-condition is not projectively invariant unless the projective
factor is a Hamel function. Therefore, the second CC-condition gives,
for a Finsler metric of constant curvature, the class of projectively
equivalent Finsler metrics that also have constant curvature. 

Beltrami's theorem states that a Riemannian metric projectively
equivalent to a Riemannian metric of constant curvature has also constant
curvature, \cite[vol. IV, p.19]{Sp}. A classic proof of Beltrami's theorem uses the projective
Weyl tensor if $n >2$ and the Liouville (Cotton) tensor if $n=2$, see
\cite{Mat} for a recent survey on the projective geometry of affine
sprays. 

In this work we prove the following general version of Beltrami's
theorem, in Finslerian setting,  for dimension $n>1$. 

\emph{If two Finsler metrics over the same connected manifold, of
  dimension $n>1$, have projectively equivalent geodesic sprays and one is of constant flag curvature
then the other is also of contant curvature if and only if the
projective factor is a Hamel function  (Theorem \ref{Beltrami Theorem}).}

For the proof of Theorem \ref{Beltrami Theorem}, we study the
projective invariance of the three CC-conditions proposed in Theorem
\ref{CONSTANT:FC}; the first two conditions for $n>2$ and the last two
conditions for $n=2$. The projective invariant $2$-form from the third
CC-condition can be viewed as the Cotton (Liouville) tensor of the
spray. The Cotton tensor has been introduced very recently by Crampin
in the geometry of a spray, \cite{Crampin19}, as an obstruction for
$2$-dimensional sprays  to be projectively $R$-flat.

In dimension $n>2$, the Finslerian version of Beltrami's theorem
has also been proved using a Weyl-type curvature tensor that characterises
Finsler metrics of constant curvature in \cite[Theorem 5.2]{BC18} and an
algebraic characterisation for Finsler metrics of constant curvature in \cite{BF19}.

\section{Isotropic sprays}
We start this section with a brief overview of the geometric setting
associated to a spray.  We consider $M$ a smooth, $n$-dimensional, connected manifold, with
$n >1$.  $TM$ is the tangent bundle and $T_0M=TM-\{0\}$ is the slit
tangent bundle of $M$. On $TM$  there are two canonical structures
that we will use further on, the Liouville (dilation) vector field $\mathcal{C}$ and
the vertical endomorphism $J$.

A \emph{spray} $S$ is a special vector field on $T_0M$ such that $JS=\mathcal{C}$ and
$[\mathcal{C},S]=S$. A \emph{geodesic} of a spray is a smooth curve on
$M$ whose velocity is an integral curve of the spray.
 
 An orientation preserving reparametrisation
 $t\mapsto \widetilde{t}(t)$ of the geodesics of a spray $S$
leads to the geodesics of a new spray $\widetilde{S}=S-2P\mathcal{C}$,  where $P\in
 C^\infty(T_0M)$ is positively 1-homogeneous. 
\begin{defn}
Two sprays $S$ and $\widetilde{S}$ are \emph{projectively related} if their geodesics coincide up to an orientation preserving reparametrisation.
\end{defn}

In this work we will use the Fr\"olicker-Nijenhuis formalism to
associate geometric data to a  spray $S$ \cite{Grifone72,
  Gr, SLK, Youssef}. The first  of them is a canonical nonlinear
connection, that determines a horizontal and a vertical projector:
\begin{equation*}
h=\frac{1}{2}(\operatorname{Id}-[S,J]), \quad v=\frac{1}{2}(\operatorname{Id}+[S,J]).
\end{equation*}

The \emph{Jacobi  endomorphism}  and the \emph{curvature} of (the
nonlinear connection determined by) $S$ are defined by
\begin{equation}\label{R-curv-tensor}
\Phi:=v\circ[S,h], \ R:=\frac{1}{2}[h,h],
\end{equation}
respectively. They are related by
\begin{equation}\label{relation:phi-R}
3R=[J,\Phi], \quad \Phi=i_SR.
\end{equation} 
The Ricci scalar $\rho\in C^\infty(T_0M)$ is  given by
\begin{equation}
\rho:=\frac{1}{n-1}\operatorname{Tr}(\Phi),
\end{equation} 
where ``Tr'' means semi-basic trace (see, e.g., \cite[p.134]{Youssef}). 
\begin{defn}
A spray $S$ is said to be \emph{isotropic} if there exists a semi-basic 1-form $\alpha\in\Lambda^1(T_0M)$ such that the Jacobi endomorphism can be written as follows: 
\begin{equation}\label{phi}
\Phi=\rho J-\alpha\otimes \mathcal{C}.
\end{equation}
\end{defn}
For an isotropic spray $S$, due to \eqref{relation:phi-R}, we have that $\Phi(S)=0$ and
hence $\rho=i_S\alpha.$

The formulae \eqref{relation:phi-R} allow to reformulate the
isotropy condition \eqref{phi} in terms of the curvature tensor $R$.
\begin{lem} \label{lem:iso-R}
A spray is isotropic if and only if its curvature tensor is
of the form
\begin{equation}\label{isotropic Curvature}
  R= \xi \wedge J - d_{J}\xi \otimes {\mathcal C},
\end{equation} 
for a semi-basic  $1$-form $\xi$ on $T_0M$. 
\end{lem}
\begin{proof}
Assume first that $S$ is isotropic and hence its Jacobi
endomorphism $\Phi$ is given by formula \eqref{phi}. Then, using
formulae (18) and (22) in \cite{Youssef}, we find 
\begin{equation}\label{R}
\begin{aligned}
3R =[J, \Phi] =[J, \rho J - \alpha \otimes {\mathcal C}]  =  (d_{J}\rho +\alpha)\wedge J -d_{J}\alpha\otimes {\mathcal C},
 \end{aligned}
 \end{equation}
which gives the desired equality with the  the semi-basic $1$-form 
\begin{equation}\label{xi}
\xi:= \frac{1}{3}\left(\alpha+d_{J}\rho\right).
\end{equation}
Conversely, assume that the curvature tensor of $S$ can be given by
formula \eqref{isotropic Curvature}.  Then 
\begin{eqnarray*}\label{isr}
\Phi=i_SR  = i_S\left( \xi \wedge J - d_{J}\xi \otimes {\mathcal
    C}\right) =i_S\xi  J-(i_Sd_J\xi+\xi)\otimes\mathcal{C}.
\end{eqnarray*} 
Taking semi-basic trace and using formula (26) in \cite{Youssef} we
find that $\operatorname{Tr}(\Phi)= (n-1) i_S\xi$. Thus the Ricci scalar is
given by $\rho=i_S\xi$, which concludes the proof. 
\end{proof}
The above characterisation of isotropic sprays
has been used by Crampin in \cite{Crampin07, Crampin19}. We will call the semi-basic $1$-form $\xi$ in formula \eqref{isotropic
  Curvature} the \emph{curvature $1$-form} of the isotropic spray.
\begin{lem} [Differential Bianchi identity for the curvature $1$-form] \label{lem:Bianchi}
In dimension $n>2$, the curvature $1$-form of an isotropic spray satisfies $d_h\xi=0$.
\end{lem}
\begin{proof}
If $S$ is any spray over $M$, then we have  $[h,{\mathcal C}]=0$,
$[h,J]=0$ (the induced connection is homogenous and torsion-free),
$[h,R]=0$ (Bianchi identity); see, e.g., \cite{Grifone72} and \cite{SLK}. 
Now suppose that $S$ is isotropic. Then, using the just mentioned
properties of the induced connection and applying formula (22) in
\cite{Youssef}, we get
\begin{eqnarray*}
0=[h,R]=[h, \xi\wedge J - d_J\xi \otimes {\mathcal C}] = d_h\xi \wedge
  J - d_hd_J\xi \otimes {\mathcal C}.
\end{eqnarray*}
Applying formulae (26) and (27) in \cite{Youssef}, we take the
semi-basic trace of the right-hand side to obtain 
\begin{eqnarray}
(n-2)d_h\xi - i_Sd_hd_J\xi =0. \label{trhr}
\end{eqnarray}
Acting by the operator $i_S$ on both sides of the last equality, we
find that $(n-2)i_Sd_h\xi=0$, whence $i_Sd_h\xi=0$. Taking this into
account, we can manipulate the second expression on the left hand side
of \eqref{trhr} as follows
\begin{eqnarray*}
- i_Sd_hd_J\xi = i_Sd_Jd_h\xi = - d_Ji_Sd_h\xi + {\mathcal L}_{JS}d_h\xi
  + i_{[J,S]}d_h\xi =  {\mathcal L}_{{\mathcal C}}d_h\xi  +
  i_{\operatorname{Id}-2v}d_h\xi = d_h\xi + 2d_h\xi = 3d_h\xi. 
\end{eqnarray*}
This concludes the proof.
\end{proof}
The semi-basic $2$-form $d_h\xi$ appears in \cite{Crampin07, Crampin19} as an obstruction for
isotropic sprays to be projectively $R$-flat. For $n>2$
this obstruction is trivial.

\section{Finsler metrics of constant flag curvature and their
  projective deformation}
In this section we provide three necessary and sufficient conditions
(CC-conditions: \eqref{isotropic1}, \eqref{dJ:xi} and \eqref{dh:xi}) for a
Finsler metric to have constant curvature (Theorem
\ref{CONSTANT:FC}). By studying the projective invariance of these
CC-conditions, we arrive at a Finslerian
version of Beltrami's theorem for dimension $n>1$ (Theorem \ref{Beltrami Theorem}).

A Finsler metric $F$ determines a unique spray $S$ on $T_0M$, called
its geodesic (or canonical) spray (see, e.g., \cite[p. 541]{SLK}). We
derive the geometric data of a Finsler metric from its geodesic spray.

A Finsler metric has \emph{scalar flag curvature} if there exists a
$0$-homogeneous function $\kappa\in C^\infty(T_0M)$ such that the
Jacobi endomorphism can be expressed as follows:
\begin{equation}\label{Jacobi:R curvature}
\Phi=\kappa(F^2J-Fd_JF\otimes\mathcal{C}).
\end{equation} 
For a Finsler metric of
scalar flag curvature the geodesic spray is isotropic; 
the Ricci scalar $\rho$ and the semi-basic $1$-form $\alpha$ are
given by
\begin{equation} \label{rak}
\rho=\kappa F^2\ \textrm{and} \  \alpha=\kappa Fd_JF,
\end{equation}
respectively. In view of Lemma \ref{lem:iso-R} and formulae
\eqref{rak} and \eqref{xi}, a Finsler metric has scalar flag
curvature if and only if its curvature tensor is of the form
\eqref{isotropic Curvature} with curvature $1$-form 
\begin{equation}\label{xik}
\xi=\frac{1}{3F}d_J\left(\kappa F^3\right) = \frac{F^2}{3}d_J\kappa + \kappa Fd_JF.
\end{equation}
If the Finsler metric $F$ has constant curvature (i.e., the function
$\kappa$ is constant), then $\xi=\kappa Fd_JF$, $d_J\xi=0$, so its curvature
tensor reduces to
\begin{equation}\label{CFC:R}
R=\kappa Fd_JF\wedge J.
\end{equation}

\begin{thm} [Finslerian version of Schur's lemma for dimension $n> 1$] \label{CONSTANT:FC}
Let $S$ be the geodesic spray of a Finsler metric $F$. Then $F$ has
constant curvature if and only if
\begin{eqnarray}\label{isotropic1}
\text{ S is isotropic } & \text{(this is always true for
  n=2);}
\end{eqnarray}
and the curvature 1-form satisfies
\begin{eqnarray}
\label{dJ:xi}
d_J\xi=0; &  \\
\label{dh:xi}
d_h\xi=0 & \text{ (this is always true  for $n >2$).}
\end{eqnarray}
\end{thm}
\begin{proof}
Assume first that $F$ is of constant curvature. Then, as we just have
seen, the geodesic spray is isotropic
and $d_J\xi=0$. It remains to show that \eqref{dh:xi} is also
satisfied.   Since $d_hF=0$ and $[h,J]=0$, we have 
$d_h\xi=\kappa Fd_hd_JF=\kappa Fd_Jd_hF=0$, as wanted.

Conversely,  suppose that the three CC-conditions are satisfied. Then \eqref{isotropic1} and Lemma
8.2.2 in \cite{Sh} (or Proposition 9.4.9 in \cite{SLK}) imply that $F$
has scalar flag curvature $\kappa$, and hence the curvature $1$-form
is given by \eqref{xik}. 
To prove the constancy of $\kappa$, it is sufficient to show that
$d_J\kappa=0$ and $d_h\kappa=0$.

By our condition \eqref{dJ:xi}, 
\begin{eqnarray*}
0=d_J\xi=\frac{2F}{3}d_JF\wedge d_J\kappa + Fd_J\kappa \wedge d_JF=\frac{F}{3}d_J\kappa \wedge d_JF,
\end{eqnarray*}
 whence 
\begin{equation}\label{dJ:xi=0}
d_J\kappa\wedge d_JF=0.
\end{equation}
Applying $i_S$ to both sides of \eqref{dJ:xi=0}, we get
$ (i_Sd_J\kappa) d_JF- (i_Sd_JF) d_J\kappa ={\mathcal C}(\kappa)d_JF -
{\mathcal C}(F) d_J\kappa = -F d_J\kappa$, therefore $d_J\kappa=0$,
and hence $\xi =\kappa F d_JF$. Now, by condition \eqref{dh:xi}, 
\begin{eqnarray*}
0=d_h\xi=d_h(\kappa F d_JF) =Fd_h\kappa \wedge d_JF + \kappa Fd_hd_JF=Fd_h\kappa \wedge d_JF,
\end{eqnarray*}
whence 
\begin{equation}\label{dh:xi=0}
d_h\kappa\wedge d_JF=0.
\end{equation}
Continuing as above, we find that 
\begin{equation}\label{iSdh:xi}
S(\kappa)d_JF-Fd_h\kappa=0.
\end{equation}
Since 
\begin{eqnarray*}
d_J{\mathcal L}_S\kappa={\mathcal L}_Sd_J\kappa + d_{[J,S]}\kappa =
  d_{h-v}\kappa = d_h\kappa,
 \end{eqnarray*}
we can rewrite \eqref{iSdh:xi} as follows:
 \begin{equation*}
  S(\kappa)d_JF-Fd_JS(\kappa)=0 \Leftrightarrow  S(\kappa)\frac{d_JF}{F^2}-\frac{1}{F}d_JS(\kappa)=0\Leftrightarrow d_J\left(\frac{S(\kappa)}{F}\right)=0.
 \end{equation*}
Arguing by contradiction, suppose that  $S(\kappa)$ is a nonzero
function. Since $d_J\kappa=0$, $\kappa$ is a vertical lift, i.e.,
$\kappa=f^v$, for some smooth function $f$ on $M$. Then
$S(\kappa)=S(f^v)=f^c$ and, by the above equality,
$d_J(f^c/F)=0$. Here $f^c$ denotes the complete lift of $f\in
C^{\infty}(M)$ (see, e.g.,  \cite[Definition 4.1.3]{SLK} for
the definitions of vertical and complete lifts of functions). Thus,
$f^c/F=h^v$, and $F=f^c/h^v$, for some smooth function $h$ on
$M$. Since $f^c$ is fiberwise linear, $h^v$ is fiberwise constant, it
follows that $F$ is a fiberwise linear function. We arrived at a
contradiction, hence $S(\kappa)=0$. Going back to formula
\eqref{iSdh:xi}, we obtain $d_h\kappa=0,$ which
assures that $\kappa$ is constant and therefore the Finsler metric has
constant curvature. 

It is well known that any two-dimensional spray manifold is isotropic
(see, e.g., \cite[Lemma 8.1.10]{Sh}), so in this case the first CC-condition is
automatically satisfied.

According to Lemma \ref{lem:Bianchi}, in dimension $n>2$, the
curvature $1$-form of an isotropic spray also automatically satisfies the
third CC-condition. 
\end{proof}
The first two CC-conditions provide an equivalent characterisation for Finsler metrics of
\emph{isotropic curvature} (scalar curvature does not depend on the
fiber coordinates) that were studied in \cite{LS18}. These conditions
were used to define a Weyl-type curvature tensor in \cite[(4.1)]{BC18}
that characterises Finsler metrics of constant curvature in dimension $n>2$. 

When $n>2$, there are alternative proofs of Theorem \ref{CONSTANT:FC}
in Finsler geometry, \cite[Proposition 26.1]{M86}, \cite[Theorem 9.4.11]{SLK}. 

Next, we study the projective invariance of the three
CC-conditions. The isotropy condition is known to be invariant and 
we prove that the third CC-condition is also invariant. As to the
second CC-condition, we show that it is invariant only for those projective
deformations $P$ satisfying the Hamel equation
\begin{eqnarray}
d_hd_JP=0. \label{Hamel}
\end{eqnarray}
A $1^+$-homogeneous, smooth function $P$ on $T_0M$, which satisfies the equation
\eqref{Hamel} is called a \emph{Hamel function}.

\begin{thm}[Finslerian version of Beltrami's theorem for dimension $n>1 $] \label{Beltrami Theorem}
If two Finsler metrics over the same connected manifold
have projectively equivalent geodesic sprays and one is of constant
flag curvature then the other is also of constant curvature if and only if the
projective factor is a Hamel function.
\end{thm}
\begin{proof}

Let $S$ and $\widetilde{S}$ be the geodesic sprays of two
projectively related Finsler metrics. Then
$\widetilde{S}=S-2P{\mathcal C}$, for a $1^{+}$-homogeneous, smooth
function  $P$ on $T_0M$. 

The isotropy condition is invariant under projective deformations and
therefore both $S$ and $\widetilde{S}$ are isotropic. According to
\cite[Proposition 4.4]{BM12}, the corresponding curvature $1$-forms $\xi$ and
$\widetilde{\xi}$ are related by 
\begin{equation}\label{sfc-general}
\widetilde{\xi}=\xi+Pd_JP-d_{h}P.
\end{equation}
If we apply $d_J$ to both sides of this equality, we have
\begin{equation} \label{pr:djxi}
d_J\widetilde{\xi}=d_J\xi - d_Jd_hP.
\end{equation}
To prove the projective invariance of the third CC-condition, 
we use that the horizontal projectors
are related by $\widetilde{h}=h-PJ - d_JP\otimes
{\mathcal C}$, see \cite[(4.8)]{BM12}, and hence
\begin{equation}\label{dhxi1}
\begin{aligned}
d_{\widetilde{h}}\widetilde{\xi} &
=d_{\widetilde{h}}\left(\xi+Pd_JP-d_hP\right) =
d_{\widetilde{h}}\xi+d_{\widetilde{h}}Pd_JP-d_{\widetilde{h}}d_hP \\ 
& = d_{h-PJ-d_JP\otimes\mathcal{C}}\xi +
d_{h-PJ-d_JP\otimes\mathcal{C}}Pd_JP-d_{h-PJ-d_JP\otimes\mathcal{C}}d_hP
\\ & = d_h\xi - Pd_J\xi-d_JP\wedge\xi+d_hP\wedge d_JP+Pd_hd_JP - d_RP
+ Pd_Jd_hP+d_JP\wedge d_hP \\ &=d_h\xi-Pd_J\xi+\xi\wedge
d_JP-d_RP+Pd_Jd_hP+Pd_hd_JP \\ &=d_h\xi-Pd_J\xi+\xi\wedge d_JP-d_RP. 
\end{aligned}
\end{equation}
Now, using the form \eqref{isotropic Curvature} of the
 tensor $R$, we have
\begin{eqnarray*}
d_RP=d_{\xi\wedge J-d_J\xi\otimes\mathcal{C}}P=\xi\wedge d_JP-Pd_J\xi,
\end{eqnarray*}
which together with \eqref{dhxi1} allows us to conclude that
\begin{equation} \label{pr:dhxi}
d_{\widetilde{h}}\widetilde{\xi}=d_h\xi.
\end{equation}
According to Theorem \ref{CONSTANT:FC}, the Finsler metric $F$ has
constant curvature if and only if the three CC-conditions are
satisfied: $S$ is isotropic, $d_J\xi=0$ and $d_h\xi=0$. Similarly, the
projectively related Finsler metric $\widetilde{F}$ has constant
curvature if and only if it satisfies the corresponding three CC-conditions. In
view of formulae \eqref{pr:djxi} and \eqref{pr:dhxi}  this is true if
anf only if $d_hd_JP=0$, which means that $P$ is a Hamel function. 
\end{proof}
Formula \eqref{pr:dhxi} shows that the curvature $2$-form $d_h\xi$ is a
projective invariant of isotropic sprays. According to Lemma \ref{lem:Bianchi}, this
curvature $2$-form vanishes
in dimension $n>2$ and therefore it is a useful projective
invariant for $2$-dimensional sprays, which are always isotropic. This corresponds to the
Liouville (Cotton) projective invariant for affine sprays in
$2$-dimensional Riemannian geometry, \cite{Mat}. In Finsler geometry,
Crampin \cite{Crampin19} defines a projective Cotton
tensor for arbitrary sprays, which reduces to the curvature $2$-form
$d_h\xi$ for isotropic sprays.
 
\subsection*{Acknowledgments} We express our thanks to Mike Crampin
and J\'ozsef Szilasi for their comments and
suggestions on this work. We owe this version of the manuscript to the
carefull reading and editorial suggestions of J\'ozsef Szilasi. 

This work has been supported by Ministry of
Research and Innovation within Program 1 -- Development of the national
RD system, Subprogram 1.2 – Institutional Performance -- RDI excellence
funding projects, Contract no.34PFE/19.10.2018.

\end{document}